\documentclass[11pt]{article}

\usepackage[a4paper,margin=1in]{geometry}
\usepackage[T1]{fontenc}
\usepackage[utf8]{inputenc}
\usepackage{lmodern}
\usepackage{microtype}
\usepackage{amsmath,amssymb,amsthm,mathtools}
\usepackage{comment}
\usepackage{mathrsfs}
\usepackage{enumitem}
\usepackage{xcolor}
\usepackage{hyperref}
\usepackage[nameinlink,capitalise]{cleveref}
\hypersetup{colorlinks=true,linkcolor=blue,citecolor=blue,urlcolor=blue}

\newtheorem{theorem}{Theorem}[section]
\newtheorem{proposition}[theorem]{Proposition}
\newtheorem{lemma}[theorem]{Lemma}
\newtheorem{corollary}[theorem]{Corollary}
\theoremstyle{definition}
\newtheorem{definition}[theorem]{Definition}
\newtheorem{remark}[theorem]{Remark}
\newtheorem{conjecture}[theorem]{Conjecture}

\newcommand{\F}{\mathbb F}
\newcommand{\Fp}{\mathbb F_p}
\newcommand{\Fpp}{\mathbb F_{p^2}}

\newcommand{\PP}{\mathbb P}

\newcommand{\rdeg}{\operatorname{rdeg}}
\newcommand{\Res}{\operatorname{Res}}
\newcommand{\squarep}{\square_p}

\title{A complete characterization of a family of permutation trinomials over $\mathbb F_{p^2}$}
\author{Marco Timpanella}
\date{}

\begin{document}
\maketitle

\begin{abstract}
Let $p>3$ be a prime and let
\[
 f_{\lambda_1,\lambda_2}(x)
 =x^{p^2-p+1}+\lambda_1x^{p^2}+\lambda_2x^{2p-1}
 \in \Fpp[x].
\]
We determine all pairs $(\lambda_1,\lambda_2)\in(\Fpp)^2$ for which
$f_{\lambda_1,\lambda_2}$ is a permutation polynomial of $\Fpp$.
The final classification consists of three explicit families. The first one is the
binomial case $\lambda_1=0$. The other two are obtained from the condition
$\lambda_2=c\lambda_1^3$, with $c\in\Fp^{*}$, and are defined by two
simple equations involving the norm $\lambda_1^{p+1}$. The proof is based on the
AGW criterion and on the study of a quartic curve naturally associated with the
rational function induced on the unit circle $\mu_{p+1}$.
\end{abstract}

\section{Introduction}

Permutation polynomials over finite fields are a classical topic and continue to play an
important role because of their applications to finite geometry, coding theory, and cryptography. We refer to \cite{HouSurvey} for a survey on permutation polynomials, and to \cite{BartoliSurveys} for a survey on the use of algebraic varieties in the study of
relevant functions over finite fields.

One of the most studied problems is the classification of sparse permutation polynomials, especially
trinomials with a specific shape. In particular, permutation trinomials with Niho-type exponents over $\F_{q^2}$, that is polynomials of the form
\[
 F(x)=x+\alpha x^{s_1(q-1)+1}+\beta x^{s_2(q-1)+1},
\]
where $s_1,s_2\in\mathbb Z$ and $\alpha,\beta\in\F_q$, have been widely investigated;
see for instance
\cite{BartoliTimpanella2021,BartoliConj,BartoliTimpanella,BaiXia,RaiGupta,
BartoliPalStanica,GrasslOzbudakOzkayaTemur,LiHellesethNew}.

In this paper we study the family
\begin{equation}\label{eq:fdef}
 f_{\lambda_1,\lambda_2}(x)
 =x^{p^2-p+1}+\lambda_1x^{p^2}+\lambda_2x^{2p-1},
 \qquad \lambda_1,\lambda_2\in\Fpp,
\end{equation}
viewed as polynomials over $\Fpp$. This is the case $q=p$ of the more general Niho-type family
\[
 x^{(p-1)q+1}+\lambda_1x^{pq}+\lambda_2x^{q+p-1},
\]
which has been studied in several recent papers under different assumptions. In
\cite{Hou}, Hou completely determined the permutation properties of this polynomial over
finite fields of characteristic $3$. Later, Bai and Xia~\cite{BaiXia} studied the case
$\lambda_1=1$, $\lambda_2=-1$ over $\F_{q^2}$, with $q=p^k$ and
$p\in\{3,5\}$, proving that it is a permutation polynomial if and only if $k$ is even.
In characteristic $5$, Gupta and Rai~\cite{RaiGupta2} considered the family with the
assumption $\lambda_2=1$ and showed that, for $k>1$, it permutes $\F_{q^2}$ precisely
when $\lambda_1=-1$ and $k$ is even. In a subsequent paper~\cite{RaiGupta}, the same
authors treated the case of characteristic $7$, proving that if $\lambda_2=1$ then the
polynomial is a permutation of $\F_{q^2}$ if and only if either
$\lambda_1=-3$ and $k=1$, or $\lambda_1=-1$ and $k=2$. They also proved that, for
$p>3$ and $k=1$, the polynomial with $\lambda_2=1$ is a permutation polynomial if and
only if $\lambda_1=-3$.

For primes $p>7$ and integers $k>1$, the authors of~\cite{RaiGupta} conjectured that if
$q=p^k$, $\lambda_2=1$, and $\lambda_1\in\F_q^*$, then the corresponding trinomial
permutes $\F_{q^2}$ if and only if $\lambda_1=-1$ and $k=2$. This conjecture was later
settled by Bartoli, Pal, and St\u{a}nic\u{a}~\cite{BartoliPalStanica}.

Our goal is to investigate the case $q=p$ in any characteristic greater than $3$, and allowing the coefficients to vary in
\(
 \lambda_1,\lambda_2\in\Fpp
\)
and obtain a complete classification. This gives, as a special case, the classification when
$\lambda_1,\lambda_2\in\Fp$, which is discussed at the end of the paper.

We now state the main result. We denote by $\Fp^*$ the nonzero elements of $\Fp$ and by
$\squarep$ the set of nonzero squares of $\Fp$.

\begin{theorem}[Main Theorem]\label{thm:main}
Let $p>3$ be a prime, and let $\lambda_1,\lambda_2\in\Fpp$, and $N=\lambda_1^{p+1}$. Define
$f_{\lambda_1,\lambda_2}$ by \eqref{eq:fdef}. Then $f_{\lambda_1,\lambda_2}$ is a
permutation polynomial of $\Fpp$ if and only if one of the following cases holds.
\begin{enumerate}[label=\rm(\roman*),leftmargin=1.5em]
\item
\[
 \lambda_1=0,
 \qquad
 p\equiv1\pmod3,
 \qquad
 \lambda_2^{p+1}\neq1.
\]

\item
$\lambda_1\neq0$, and there exists $c\in\Fp^*$ such that
\[
 \lambda_2=c\lambda_1^3,\qquad  c^2N^3-cN^2-1=0
 \qquad\text{and}\qquad
 1-4c^2N^3\in\squarep.
\]

\item
$\lambda_1\neq0$, and there exists $c\in\Fp^*$ such that

\[
 \lambda_2=c\lambda_1^3,\qquad   3cN+1=0
 \qquad\text{and}\qquad
-3(1-4c^2N^3)\in\squarep.
\]
\end{enumerate}
\end{theorem}

The proof follows the usual reduction from permutation polynomials of $\Fpp$ to rational
functions permuting the unit circle
\[
 \mu_{p+1}=\{t\in\Fpp^{*}:t^{p+1}=1\}.
\]
For our family, the induced rational function is
\[
 G(T)=\frac{T^3+\lambda_1^pT^2+\lambda_2^p}{1+\lambda_1T+\lambda_2T^3}.
\]

The reduced degree of $G$ gives the first division of the proof. Reduced degree $0$ and $2$
do not produce permutation polynomials. Reduced degree $1$ gives the second family in the
Main Theorem. In reduced degree $3$ we study the curve 
\[
 \frac{N(X)D(Y)-N(Y)D(X)}{X-Y}=0,
\]
where $N(T)=T^3+\lambda_1^pT^2+\lambda_2^p$ and
$D(T)=1+\lambda_1T+\lambda_2T^3$. 
 The only admissible
factorization in degree $3$ is a product of two conics exchanged by the involution
$(X,Y)\mapsto(Y,X)$, and this gives the third family in the main theorem.

The paper is organized as follows. In Section~\ref{sec:preliminaries} we recall the tools used
in the proof. In Section~\ref{sec:rational-functions} we associate the rational function
$G_{\lambda_1,\lambda_2}$ to $f_{\lambda_1,\lambda_2}$ and determine its reduced degree. In
Section~\ref{sec:degree-at-most-two} we settle the cases of reduced degree at most $2$.
Section~\ref{sec:degree-three} is devoted to the quartic curve arising in reduced degree $3$,
to the analysis of its possible factorizations, and to the proof of the main theorem.

\section{Preliminaries}\label{sec:preliminaries}
Throughout the paper $p$ denotes an odd prime and we let
\[
 \mu_{p+1}=\{t\in\F_{p^2}^*:t^{p+1}=1\}.
\]

An important tool for us will be the usual special case of the Akbary--Ghioca--Wang criterion, see \cite[Corollary 2.4]{AkbaryGhiocaWang}.
\begin{proposition}[AGW criterion]\label{prop:AGW}
Let $q$ be a prime power, let $d\mid(q-1)$, let $r\ge 1$, and let $h\in\F_q[x]$.
Set
\[
 F(x)=x^r h\bigl(x^{(q-1)/d}\bigr).
\]
Then $F$ permutes $\F_q$ if and only if
\begin{enumerate}[label=\rm(\roman*),leftmargin=1.5em]
\item $\gcd\bigl(r,(q-1)/d\bigr)=1$,
\item $h(\zeta)\ne 0$ for every $\zeta\in\mu_d$,
\item the map
\[
 \zeta\longmapsto \zeta^r h(\zeta)^{(q-1)/d}
\]
permutes $\mu_d$.
\end{enumerate}
\end{proposition}

Next we recall the notion of equivalence for rational functions.
\begin{definition}
Two rational functions $f,g\in\F_p(X)$ are called \emph{equivalent} if
there exist degree-one rational functions $\varphi,\psi\in\F_p(X)$ such that
\[
 g=\varphi\circ f\circ\psi.
\]
\end{definition}
Equivalent rational functions induce the same permutation behavior on
$\PP^1(\F_p)$.

We need the classification of permutation rational functions of degree $2$ and $3$.
The degree-$2$ case is classical and is recalled, for example, in
\cite[p.~3]{DingZieveLowDegree}; the degree-$3$ statement is due to Ferraguti and Micheli
\cite[Theorem 1.3]{FerragutiMicheli}.

\begin{proposition}[Permutation rational functions of degree $2$ and $3$]
\label{prop:PRF23}
Let $q$ be a prime power.
\begin{enumerate}[label=\rm(\arabic*),leftmargin=1.5em]
\item A degree $1$ rational function of the form $(aX+b)/(cX+d)$ with $a,b,c,d\in\F_q$ permutes $\PP^1(\F_q)$ if and only if $ad-bc\ne0$.
\item A degree-two rational function in $\F_q(X)$ permutes $\PP^1(\F_q)$ if and only if
$q$ is even and the function is equivalent to $X^2$.
\item Assume $3\nmid q$. A degree-three rational function in $\F_q(X)$ permutes
$\PP^1(\F_q)$ if and only if it is equivalent to one of the following:
\begin{enumerate}[label=\rm(\alph*),leftmargin=1.5em]
\item $X^3$, if $q\equiv 2\pmod 3$;
\item $\nu^{-1}\circ X^3\circ\nu$, if $q\equiv 1\pmod 3$, where for some
$\delta\in\F_{q^2}\setminus\F_q$ one has
\[
 \nu(X)=\frac{X-\delta^q}{X-\delta},
 \qquad
 \nu^{-1}(X)=\frac{\delta X-\delta^q}{X-1}.
\]
\end{enumerate}
\end{enumerate}
\end{proposition}

Finally we recall the version of the Aubry--Perret bound that will be used in the investigation of a quartic associated to $f_{\lambda_1,\lambda_2}$. 

\begin{proposition}[Aubry--Perret bound, \cite{AubryPerret}]\label{prop:AP}
Let \(\mathcal X\subseteq\PP^2\) be an absolutely irreducible projective plane
curve defined over \(\Fp\), of degree \(d\). Then
\[
 \#\mathcal X(\Fp)
 \ge
 p+1-(d-1)(d-2)\sqrt p.
\]
\end{proposition}

\section{Rational functions associated to $f$}\label{sec:rational-functions}

In this section we apply the AGW criterion and introduce the rational function whose
permutation behaviour controls the polynomial $f_{\lambda_1,\lambda_2}$.

\begin{proposition}\label{prop:agw-reduction}
Let $f_{\lambda_1,\lambda_2}$ be defined by \eqref{eq:fdef}. Then
$f_{\lambda_1,\lambda_2}$ permutes $\Fpp$ if and only if the following two conditions hold:
\begin{enumerate}[label=\rm(\alph*),leftmargin=1.5em]
\item
\begin{equation}\label{eq:denominator-condition}
 1+\lambda_1v+\lambda_2v^3\neq0
 \qquad\text{for every }v\in\mu_{p+1};
\end{equation}
\item the rational function
\begin{equation}\label{eq:Gdef}
 G_{\lambda_1,\lambda_2}(T)=
 \frac{T^3+\lambda_1^pT^2+\lambda_2^p}{1+\lambda_1T+\lambda_2T^3}
\end{equation}
permutes $\mu_{p+1}$.
\end{enumerate}
\end{proposition}

\begin{proof}
For $x\in\Fpp^{*}$ put $t=x^{p-1}$. Then $t\in\mu_{p+1}$ and
\[
 x^{p^2-p+1}=x(x^{p-1})^p=xt^p,
 \qquad
 x^{p^2}=x,
 \qquad
 x^{2p-1}=x(x^{p-1})^2=xt^2.
\]
Thus
\[
 f_{\lambda_1,\lambda_2}(x)=x\bigl(t^p+\lambda_1+\lambda_2t^2\bigr).
\]
Since $t^p=t^{-1}$ on $\mu_{p+1}$, this is
\[
 f_{\lambda_1,\lambda_2}(x)=x h(t),
 \qquad
 h(t)=t^{-1}+\lambda_1+\lambda_2t^2.
\]
Applying \cref{prop:AGW} with $q=p^2$, $d=p+1$, and $r=1$, we see that
$f_{\lambda_1,\lambda_2}$ permutes $\Fpp$ if and only if $h$ has no zeros on
$\mu_{p+1}$ and the map
\[
 t\longmapsto t h(t)^{p-1}
\]
permutes $\mu_{p+1}$. Now
\[
 t h(t)^{p-1}=t\frac{h(t)^p}{h(t)}.
\]
Using $t^p=t^{-1}$ we have
\[
 h(t)^p=t+\lambda_1^p+\lambda_2^pt^{-2},
\]
whence
\[
 t\frac{h(t)^p}{h(t)}=
 \frac{t^3+\lambda_1^pt^2+\lambda_2^p}{1+\lambda_1t+\lambda_2t^3},
\]
which is \eqref{eq:Gdef}. Finally, $th(t)=1+\lambda_1t+\lambda_2t^3$, so the no-zero
condition for $h$ is exactly \eqref{eq:denominator-condition}.
\end{proof}

Put
\[
 N(T)=T^3+\lambda_1^pT^2+\lambda_2^p,
 \qquad
 D(T)=1+\lambda_1T+\lambda_2T^3.
\]

By direct computation the resultant of $N$ and $D$ with respect to $T$ is
\begin{align}\label{eq:resultant-ND}
\operatorname{Res}_T(N,D)
={}&
-\lambda_1^{3p}\lambda_2
-\lambda_1^{2p+2}\lambda_2^{p+1}
+2\lambda_1^{p+1}\lambda_2^{2p+2} \notag\\
&-\lambda_1^{p+1}\lambda_2^{p+1}
-\lambda_1^{p+1}
-\lambda_2^{3p+3}
+3\lambda_2^{2p+2} \notag\\
&-\lambda_1^3\lambda_2^p
-3\lambda_2^{p+1}
+1.
\end{align}

We denote by $\rdeg G_{\lambda_1,\lambda_2}$ the reduced degree of $G_{\lambda_1,\lambda_2}$, that is the degree of $N/D$ after cancellation of
common factors. The following proposition deals with the four possible reduced degrees.

\begin{proposition}\label{prop:reduced-degree}
Let $G=G_{\lambda_1,\lambda_2}$. Then the following hold.
\begin{enumerate}[label=\rm(\roman*),leftmargin=1.5em]
\item $\rdeg G=0$ if and only if
\[
 \lambda_1=0,
 \qquad
 \lambda_2^{p+1}=1.
\]

\item $\rdeg G=1$ if and only if $\lambda_1\neq0$ and
\begin{equation}\label{eq:E1}
 \lambda_1^{2p}\lambda_2-\lambda_1\lambda_2^{p+1}+\lambda_1=0,
\end{equation}
\begin{equation}\label{eq:E2}
 \lambda_1^{2p+1}\lambda_2
 -\lambda_1^p\lambda_2^{p+2}
 +\lambda_1^p\lambda_2+\lambda_1^2=0.
\end{equation}

\item $\rdeg G=2$ if and only if the resultant $\Res_T(N,D)=0$ and neither of
the previous two cases occurs.

\item $\rdeg G=3$ if and only if $\Res_T(N,D)\neq0$.
\end{enumerate}
\end{proposition}

\begin{proof}
The reduced degree of \(G_{\lambda_1,\lambda_2}\) is
\(
3-\deg(\gcd(N,D)).
\)

First, the reduced degree is zero if and only if \(N\) and \(D\) are
proportional. Thus there exists \(c\in\mathbb F_{p^2}^{\star}\) such that
\[
 N(T)=cD(T).
\]
Comparing the coefficients of \(T^3,T^2,T\), and the constant term gives
\[
 1=c\lambda_2,\qquad
 \lambda_1^p=0,\qquad
 0=c\lambda_1,\qquad
 \lambda_2^p=c.
\]
Hence \(\lambda_1=0\) and
\[
 1=c\lambda_2=\lambda_2^{p+1}.
\]
Conversely, if \(\lambda_1=0\) and \(\lambda_2^{p+1}=1\), then
\[
 N(T)=T^3+\lambda_2^p
      =
      \lambda_2^p(1+\lambda_2T^3)
      =
      \lambda_2^pD(T),
\]
and so the reduced degree is zero.

We now determine when the reduced degree is one. This is equivalent to
\[
 \deg\gcd(N,D)=2.
\]
In this case \(\lambda_2\neq0\), since otherwise \(D\) has degree at most one.
Also, \(\lambda_1\neq0\). Indeed, if \(\lambda_1=0\), then
\[
 N(T)=T^3+\lambda_2^p,\qquad
 D(T)=1+\lambda_2T^3,
\]
and these two polynomials are either proportional, when
\(\lambda_2^{p+1}=1\), or coprime. Therefore they cannot have a common factor of
degree two.

Assume therefore that \(\lambda_1\lambda_2\neq0\). Let
\[
 Q(T)=T^2+uT+v
\]
be the common monic quadratic factor. Since \(N\) is monic and the leading
coefficient of \(D\) is \(\lambda_2\), there exist \(a,d\in\Fpp\) such that
\[
 N(T)=(T+a)Q(T),
 \qquad
 D(T)=(\lambda_2T+d)Q(T).
\]
Subtracting \(\lambda_2N(T)\) from \(D(T)\), we get
\[
 D(T)-\lambda_2N(T)
 =
 (d-\lambda_2a)Q(T).
\]
On the other hand,
\[
\begin{aligned}
D(T)-\lambda_2N(T)
&=
1+\lambda_1T+\lambda_2T^3
-\lambda_2(T^3+\lambda_1^pT^2+\lambda_2^p)\\
&=
-\lambda_1^p\lambda_2T^2+\lambda_1T+1-\lambda_2^{p+1}.
\end{aligned}
\]
Since \(\lambda_1\lambda_2\neq0\), the coefficient of \(T^2\) is nonzero.
Therefore
\[
 d-\lambda_2a=-\lambda_1^p\lambda_2,
\]
and comparison of the remaining coefficients gives
\[
 u=-\frac{\lambda_1}{\lambda_1^p\lambda_2},
 \qquad
 v=\frac{\lambda_2^{p+1}-1}{\lambda_1^p\lambda_2}.
\]
Now compare the coefficients in
\[
 N(T)=(T+a)(T^2+uT+v).
\]
Expanding gives
\[
 (T+a)(T^2+uT+v)
 =
 T^3+(a+u)T^2+(v+au)T+av.
\]
Thus
\[
 a+u=\lambda_1^p,\qquad
 v+au=0,\qquad
 av=\lambda_2^p.
\]
From \(a+u=\lambda_1^p\), using the value of \(u\), we obtain
\[
 a=\lambda_1^p+\frac{\lambda_1}{\lambda_1^p\lambda_2}.
\]
It remains to impose the last two coefficient equations. Substituting the above
values of \(a,u,v\), we get
\[
 v+au
 =
 -
 \frac{
 \lambda_1^{2p+1}\lambda_2
 -
 \lambda_1^p\lambda_2^{p+2}
 +
 \lambda_1^p\lambda_2
 +
 \lambda_1^2
 }
 {\lambda_1^{2p}\lambda_2^2},
\]
and
\[
 av-\lambda_2^p
 =
 -
 \frac{
 \lambda_1^{2p}\lambda_2
 -
 \lambda_1\lambda_2^{p+1}
 +
 \lambda_1
 }
 {\lambda_1^{2p}\lambda_2^2}.
\]
Consequently, \(Q(T)\) divides both \(N(T)\) and \(D(T)\) if and only if
\[
 \lambda_1^{2p}\lambda_2
 -
 \lambda_1\lambda_2^{p+1}
 +
 \lambda_1
 =
0
\]
and
\[
 \lambda_1^{2p+1}\lambda_2
 -
 \lambda_1^p\lambda_2^{p+2}
 +
 \lambda_1^p\lambda_2
 +
 \lambda_1^2
 =
0.
\]
These are precisely \eqref{eq:E1} and \eqref{eq:E2}.

Conversely, assume that \(\lambda_1\neq0\) and that \eqref{eq:E1},
\eqref{eq:E2} hold. Then \(\lambda_2\neq0\), since \eqref{eq:E1} with
\(\lambda_2=0\) would give \(\lambda_1=0\). Define
\[
 u=-\frac{\lambda_1}{\lambda_1^p\lambda_2},
 \qquad
 v=\frac{\lambda_2^{p+1}-1}{\lambda_1^p\lambda_2},
\]
and
\[
 a=\lambda_1^p-u.
\]
The two equations \eqref{eq:E1}, \eqref{eq:E2} are exactly the conditions
\[
 av=\lambda_2^p,
 \qquad
 v+au=0.
\]
Therefore
\[
 N(T)=(T+a)(T^2+uT+v).
\]
Moreover,
\[
 D(T)-\lambda_2N(T)
 =
 -\lambda_1^p\lambda_2(T^2+uT+v),
\]
and hence \(D(T)\) is also divisible by \(T^2+uT+v\). Thus
\[
 \deg\gcd(N,D)=2.
\]
Since \(\lambda_1\neq0\), \(N\) and \(D\) are not proportional, so the reduced
degree is exactly one.

Finally, the last two cases are controlled by the resultant. One has
\[
 \operatorname{Res}_T(N,D)=0
\]
if and only if \(N\) and \(D\) have a nonconstant common factor. After the cases
of common factor of degree \(3\) and \(2\) have been removed, this common factor
has degree \(1\), and the reduced degree is \(2\). If the resultant is nonzero,
then no cancellation occurs and the reduced degree is \(3\).
\end{proof}

\section{Reduced degree at most two}\label{sec:degree-at-most-two}

We now determine the permutation behaviour when $\rdeg G_{\lambda_1,\lambda_2}\le2$.

\begin{proposition}\label{prop:degree-zero-two}
If $\rdeg G_{\lambda_1,\lambda_2}=0$ or $\rdeg G_{\lambda_1,\lambda_2}=2$, then
$f_{\lambda_1,\lambda_2}$ is not a permutation polynomial of $\Fpp$.
\end{proposition}

\begin{proof}
If the reduced degree is zero, then $G_{\lambda_1,\lambda_2}$ is constant and cannot permute
$\mu_{p+1}$.

Assume now that the reduced degree is two. Choose $\beta\in\Fpp\setminus\Fp$ with
$\beta^p=-\beta$. The map
\[
 z\longmapsto \frac{z+\beta}{z-\beta}
\]
induces a bijection $\PP^1(\Fp)\to\mu_{p+1}$. Conjugating
$G_{\lambda_1,\lambda_2}$ by this bijection gives a rational function of degree $2$
with coefficients in $\Fp$ on $\PP^1(\Fp)$. By \cref{prop:PRF23}(2), no rational
function of degree $2$ over a field of odd characteristic $p$ permutes $\PP^1(\Fp)$. Hence
$G_{\lambda_1,\lambda_2}$ does not permute $\mu_{p+1}$, and the claim follows from
\cref{prop:agw-reduction}.
\end{proof}

The reduced degree-one case gives the second family of the Main Theorem.

\begin{proposition}\label{prop:degree-one-family}
Assume that $\rdeg G_{\lambda_1,\lambda_2}=1$ and let $N=\lambda_1^{p+1}$. Then $f_{\lambda_1,\lambda_2}$ is a
permutation polynomial of $\Fpp$ if and only if there exists $c\in\Fp^{*}$ such that
\[
 \lambda_2=c\lambda_1^3,\qquad c^2N^3-cN^2-1=0,
 \qquad
 1-4c^2N^3\in\squarep.
\]
\end{proposition}

\begin{proof}
By \cref{prop:reduced-degree}, we have $\lambda_1\neq0$ and \eqref{eq:E1},
\eqref{eq:E2} hold. Put
\[
 c=\frac{\lambda_2}{\lambda_1^3},
 \qquad
 N=\lambda_1^{p+1}.
\]
Then $c\in \Fpp^*$ and $N\in\Fp^*$. Dividing \eqref{eq:E1} by $\lambda_1$ gives
\[
 \lambda_1^{2p-1}\lambda_2-\lambda_2^{p+1}+1=0.
\]
Since $\lambda_2=c\lambda_1^3$ and $N=\lambda_1^{p+1}$, this becomes
\begin{equation}\label{eq:degree-one-first}
 c^{p+1}N^3-cN^2-1=0.
\end{equation}
Next divide \eqref{eq:E2} by $\lambda_1^2$. We get
\[
 \lambda_1^{2p-1}\lambda_2
 -\lambda_1^{p-2}\lambda_2^{p+2}
 +\lambda_1^{p-2}\lambda_2+1=0.
\]
Substituting $\lambda_2=c\lambda_1^3$ gives
\[
 cN^2-c^{p+2}N^4+cN+1=0.
\]
Multiplying \eqref{eq:degree-one-first} by $cN$ gives
\[
 c^{p+2}N^4=c^2N^3+cN.
\]
Therefore the previous equation is equivalent to
\begin{equation}\label{eq:degree-one-second}
 c^2N^3-cN^2-1=0.
\end{equation}
Comparing \eqref{eq:degree-one-first} and \eqref{eq:degree-one-second} yields
$c^{p+1}=c^2$. Since $c\neq0$, we get $c^{p-1}=1$, hence $c\in\Fp^*$.
Thus the reduced degree-one condition is equivalent to
\[
 \lambda_2=c\lambda_1^3,
 \qquad
 c\in\Fp^*,
 \qquad
 c^2N^3-cN^2-1=0.
\]

It remains to consider condition \eqref{eq:denominator-condition}. Let $v\in\mu_{p+1}$ and put
$w=\lambda_1v$. Then $w^{p+1}=N$ and
\[
 1+\lambda_1v+\lambda_2v^3=1+w+cw^3.
\]
Suppose that $1+w+cw^3=0$ and $w^{p+1}=N$. Taking the $p$-th power gives
\[
 1+\frac{N}{w}+c^p\frac{N^3}{w^3}=0.
\]

Since \(c\in\Fp\), by raising \(1+w+cw^3=0\) to the \(p\)-th power and using
\(w^{p+1}=N\), we get
\[
1+\frac{N}{w}+c\frac{N^3}{w^3}=0.
\]
Multiplying by \(w^3\), this becomes
\[
w^3+Nw^2+cN^3=0.
\]
On the other hand, from \(1+w+cw^3=0\) we have
\[
w^3=-\frac{1+w}{c}.
\]
Substituting this into the previous equation gives
\[
-\frac{1+w}{c}+Nw^2+cN^3=0.
\]
Multiplying by \(c\) and using
\[
c^2N^3-cN^2-1=0,
\]
we obtain
\[
cNw^2-w+cN^2=0.
\]

Conversely, suppose that
\[
cNw^2-w+cN^2=0.
\]
Reducing \(1+w+cw^3\) modulo this quadratic gives
\[
1+w+cw^3=
-\frac{w}{cN^2}(c^2N^3-cN^2-1)=0.
\]

By the above argument, condition \eqref{eq:denominator-condition} is equivalent to the non-existence of a root of
\[
 cNX^2-X+cN^2=0
\]
having norm \(N\).
The discriminant of this quadratic is
\[
 \Delta=1-4c^2N^3\in\Fp.
\]
If \(\Delta=0\), the double root \(w\) satisfies \(w^2=N\). Since
\(w\in\Fp\), this gives \(w^{p+1}=N\). If \(\Delta\) is a nonsquare in \(\Fp\), the two roots are conjugate in
\(\Fpp\setminus\Fp\). Their product is \(N\), and hence each of them has norm
\(N\). If $\Delta\in\squarep$, then the two roots are
distinct elements of \(\Fp\). Their product is \(N\). If one root \(w\) had
norm \(N\), then, since \(w\in\Fp\), we would have \(w^2=N\). The other root
would then be \(N/w=w\), contradicting the fact that the two roots are
distinct. Hence no root has norm \(N\) if \(\Delta\in\squarep\).

Consequently \eqref{eq:denominator-condition} is equivalent, in the reduced
degree-one case, to $1-4c^2N^3\in\squarep$. When this holds, $G_{\lambda_1,\lambda_2}$ is an invertible fractional linear transformation of $\PP^1(\Fpp)$. Then $G_{\lambda_1,\lambda_2}$ permutes $\mu_{p+1}$ and the result follows from
\cref{prop:agw-reduction}.
\end{proof}

\section{Reduced degree three}\label{sec:degree-three}

We now assume that \(\rdeg G_{\lambda_1,\lambda_2}=3\). We first treat the
case \(\lambda_1=0\).

\begin{proposition}\label{prop:binomial-family}
Assume that \(\lambda_1=0\). Then \(f_{\lambda_1,\lambda_2}\) is a permutation
polynomial of \(\Fpp\) if and only if
\[
 p\equiv1\pmod3
 \qquad\text{and}\qquad
 \lambda_2^{p+1}\neq1.
\]
\end{proposition}

\begin{proof}
If \(\lambda_1=0\), then
\[
 G_{0,\lambda_2}(T)
 =
 \frac{T^3+\lambda_2^p}{1+\lambda_2T^3}.
\]
If \(\lambda_2^{p+1}=1\), then by Proposition~\ref{prop:reduced-degree}
the reduced degree of \(G_{0,\lambda_2}\) is zero. Hence \(G_{0,\lambda_2}\)
is constant and cannot permute \(\mu_{p+1}\). Thus a necessary condition is
\[
 \lambda_2^{p+1}\neq1.
\]

Assume now that \(\lambda_2^{p+1}\neq1\). Then the map
\[
 z\longmapsto \frac{z+\lambda_2^p}{1+\lambda_2 z}
\]
is a projective linear transformation, since its determinant is
\[
1-\lambda_2^{p+1}\neq0.
\]
Therefore
\[
G_{0,\lambda_2}(T)=
\frac{T^3+\lambda_2^p}{1+\lambda_2T^3}
\]
permutes \(\mu_{p+1}\) if and only if the cube map \(T\mapsto T^3\) permutes
\(\mu_{p+1}\), and the denominator \(1+\lambda_2T^3\) has no zero on
\(\mu_{p+1}\).

The cube map permutes \(\mu_{p+1}\) if and only if
\[
 \gcd(3,p+1)=1,
\]
that is, if and only if
\[
 p\equiv1\pmod3.
\]
Under this hypothesis, \(\{T^3:T\in\mu_{p+1}\}=\mu_{p+1}\). If
\(\lambda_2=0\), then the denominator is identically \(1\), and the condition
\(\lambda_2^{p+1}\neq1\) is automatically satisfied. If \(\lambda_2\neq0\),
then the condition
\[
1+\lambda_2T^3\neq0
\qquad\text{for every }T\in\mu_{p+1}
\]
is equivalent to
\[
-\lambda_2^{-1}\notin\mu_{p+1}.
\]
Since \(p+1\) is even, this is equivalent to
\[
(-\lambda_2^{-1})^{p+1}\neq1,
\]
that is,
\[
\lambda_2^{p+1}\neq1.
\]
The claim follows from Proposition~\ref{prop:agw-reduction}.
\end{proof}

From now on in this section we assume \(\lambda_1\neq0\). Our aim is to show
that the only admissible reduced degree-three case gives the third family in
the Main Theorem.

Define
\begin{equation}\label{eq:Hdef}
 H_{\lambda_1,\lambda_2}(X,Y)=
 \frac{N(X)D(Y)-N(Y)D(X)}{X-Y}.
\end{equation}
For \(x\neq y\), the equality
\[
H_{\lambda_1,\lambda_2}(x,y)=0
\]
is exactly the condition
\[
G_{\lambda_1,\lambda_2}(x)=G_{\lambda_1,\lambda_2}(y),
\]
provided that both denominators are nonzero. A direct computation gives
\begin{align}\label{eq:Hexplicit}
H_{\lambda_1,\lambda_2}(X,Y)={}&
-\lambda_1^p\lambda_2X^2Y^2
+\lambda_1(X^2Y+XY^2)\notag\\
&+(1-\lambda_2^{p+1})(X^2+Y^2)\notag\\
&+(\lambda_1^{p+1}-\lambda_2^{p+1}+1)XY
+\lambda_1^p(X+Y)
-\lambda_1\lambda_2^p.
\end{align}
Clearly \(H_{\lambda_1,\lambda_2}\) is symmetric in \(X,Y\). We write
\begin{equation}\label{eq:ABCDEF}
H_{\lambda_1,\lambda_2}(X,Y)
=
AX^2Y^2+B(X^2Y+XY^2)+C(X^2+Y^2)+DXY+E(X+Y)+F,
\end{equation}
where
\begin{equation}\label{eq:coefficients-direct}
\begin{gathered}
 A=-\lambda_1^p\lambda_2,
 \qquad
 B=\lambda_1,
 \qquad
 C=1-\lambda_2^{p+1},\\
 D=\lambda_1^{p+1}-\lambda_2^{p+1}+1,
 \qquad
 E=\lambda_1^p,
 \qquad
 F=-\lambda_1\lambda_2^p.
\end{gathered}
\end{equation}
Let \(\mathcal C_{\lambda_1,\lambda_2}\) be the affine plane curve defined by
\[
H_{\lambda_1,\lambda_2}(X,Y)=0.
\]
If \(G_{\lambda_1,\lambda_2}\) permutes \(\mu_{p+1}\), then
\(\mathcal C_{\lambda_1,\lambda_2}\) has no point
\[
(x,y)\in\mu_{p+1}^2,\qquad x\neq y,
\]
at which both denominators are nonzero. Indeed, such a point would give
\[
G_{\lambda_1,\lambda_2}(x)=G_{\lambda_1,\lambda_2}(y),
\]
contradicting the injectivity of \(G_{\lambda_1,\lambda_2}\) on \(\mu_{p+1}\).

We now show that, in the permutation case, the curve \(\mathcal C_{\lambda_1,\lambda_2}\) cannot be
absolutely irreducible.

\begin{lemma}\label{lem:irreducible-impossible}
Let \(p>3\) and \(
 \rdeg G_{\lambda_1,\lambda_2}=3
\). If \(f_{\lambda_1,\lambda_2}\) is a permutation
polynomial of \(\Fpp\), then \(\mathcal C_{\lambda_1,\lambda_2}\) is not
absolutely irreducible.
\end{lemma}

\begin{proof}
Assume, by contradiction, that \(\mathcal C_{\lambda_1,\lambda_2}\) is
absolutely irreducible. Choose \(\beta\in\Fpp\setminus\Fp\) such that \(\beta^p=-\beta\), and define
\[
 \psi(U)=\frac{U+\beta}{U-\beta}
 \qquad (U\in\PP^1(\Fp)).
\]
For \(U\in\Fp\) one has
\[
 \psi(U)^p=\frac{U-\beta}{U+\beta}=\psi(U)^{-1},
\]
so \(\psi(U)\in\mu_{p+1}\). Also \(\psi(\infty)=1\). Thus \(\psi\) induces a
bijection
\[
 \PP^1(\Fp)\longrightarrow \mu_{p+1},
\]
with inverse
\[
 \psi^{-1}(t)=\beta\,\frac{t+1}{t-1}.
\]

After multiplying
\[
H_{\lambda_1,\lambda_2}
\left(
\psi(U),\psi(V)
\right)
=
H_{\lambda_1,\lambda_2}
\left(
\frac{U+\beta}{U-\beta},
\frac{V+\beta}{V-\beta}
\right)
\]
by \((U-\beta)^2(V-\beta)^2\), we obtain an affine plane curve
\(\mathcal D\). After multiplication by a nonzero scalar, \(\mathcal D\) is
defined over \(\Fp\). Moreover, the change of variables is birational, so
\(\mathcal D\) is absolutely irreducible.

By construction, every affine point \((u,v)\in\mathcal D(\Fp)\) gives a point
\[
 (x,y)=\bigl(\psi(u),\psi(v)\bigr)\in\mu_{p+1}\times\mu_{p+1}
\]
satisfying
\[
 H_{\lambda_1,\lambda_2}(x,y)=0.
\]
Moreover, \(u\neq v\) if and only if \(x\neq y\). Thus an \(\Fp\)-rational
affine point of \(\mathcal D\) off the diagonal \(U=V\) gives a point of
\(\mathcal C_{\lambda_1,\lambda_2}\) with coordinates in \(\mu_{p+1}\), off the
diagonal \(X=Y\).

Let \(\overline{\mathcal D}\) be the projective closure of \(\mathcal D\) in
\(\PP^2\). Since the polynomial \(H_{\lambda_1,\lambda_2}(X,Y)\) has degree at
most \(4\), also
\[
 \deg \overline{\mathcal D}\le 4.
\]
By the Aubry--Perret bound,
\[
 \#\overline{\mathcal D}(\Fp)
 \ge
 p+1-6\sqrt p.
\]

Observe that both the line at infinity and the diagonal meet \(\overline{\mathcal D}\) in
at most \(4\) points, and neither of these two lines is a component of \(\overline{\mathcal D}\).

So, if \(p\ge53\), one has
\[
 p+1-6\sqrt p>8,
\]
and hence \(\overline{\mathcal D}\) has an \(\Fp\)-rational affine point
\((u,v)\) off the diagonal. So, the corresponding elements
\[
 x=\frac{u+\beta}{u-\beta},
 \qquad
 y=\frac{v+\beta}{v-\beta}
\]
belong to \(\mu_{p+1}\), are distinct, and satisfy
\[
 H_{\lambda_1,\lambda_2}(x,y)=0.
\]
This is a contradiction as \(f_{\lambda_1,\lambda_2}\) is a permutation polynomial.

It remains to consider the finite set of primes not covered by the above
estimate, namely
\[
 p=5,7,11,13,17,19,23,29,31,37,41,43,47.
\]
For these primes the claim can be easily directly checked by Magma.
Thus the absolutely irreducible case is impossible for every prime \(p>3\).
\end{proof}

We next exclude the case \(\lambda_2=0\) in reduced degree three.

\begin{proposition}\label{prop:lambda2-zero-rdeg3}
Assume that \(\lambda_1\neq0\) and \(\lambda_2=0\). If
\(\rdeg G_{\lambda_1,\lambda_2}=3\), then \(f_{\lambda_1,\lambda_2}\) is not
a permutation polynomial of \(\Fpp\).
\end{proposition}

\begin{proof}
For \(\lambda_2=0\), we have
\[
 G_{\lambda_1,0}(T)
 =
 \frac{T^3+\lambda_1^pT^2}{1+\lambda_1T}
 =
 \frac{T^2(T+\lambda_1^p)}{1+\lambda_1T}.
\]
It is readily seen that \(
 \rdeg G_{\lambda_1,0}=3\) if and only if \(\lambda_1^{p+1}\neq1.\) Put
\(
 \nu=\lambda_1^{p+1}\neq 1.
\)
Then \(\mathcal C_{\lambda_1,\lambda_2}\) is defined by the affine equation
\[
\begin{aligned}
0={}&
\lambda_1^2XY(X+Y)
+\lambda_1(X^2+Y^2)  \\
&+\lambda_1(\nu+1)XY
+\nu(X+Y)=\lambda_1H_{\lambda_1,\lambda_2}
\end{aligned}
\]
and it is a cubic curve.
We prove that this cubic has no linear factor over \(\overline{\mathbb{F}}_p\). Since a reducible plane cubic always has a line component, this proves that \(\mathcal C_{\lambda_1,\lambda_2}\) is absolutely irreducible.

The homogeneous part of degree \(3\) is
\[
 \lambda_1^2XY(X+Y).
\]
Therefore, if the cubic had a linear factor, it would have one of the following forms:
\[
 X+\eta,\qquad Y+\eta,\qquad X+Y+\eta
\]
for some \(\eta\in\overline{\Fp}\).

First suppose that \(X+\eta\) divides the cubic. Substituting \(X=-\eta\), the
resulting polynomial in \(Y\) must be identically zero. A direct computation
gives
\[
-\lambda_1(\lambda_1\eta-1)Y^2
+
(\lambda_1\eta-\nu)(\lambda_1\eta-1)Y
+
\eta(\lambda_1\eta-\nu).
\]
The coefficient of \(Y^2\) gives
\[
 \lambda_1\eta=1.
\]
Substituting this into the constant term gives
\[
 \eta(1-\nu)=0.
\]
Since \(\eta=1/\lambda_1\neq0\), this forces \(\nu=1\), a contradiction. Hence there is no factor of the form \(X+\eta\).
By symmetry, there is no factor of the form \(Y+\eta\).

It remains to consider a factor of the form \(X+Y+\eta\). Substituting
\(Y=-X-\eta\), we obtain
\[
\lambda_1(\lambda_1\eta-\nu+1)X^2
+
\lambda_1\eta(\lambda_1\eta-\nu+1)X
+
\eta(\lambda_1\eta-\nu).
\]
The coefficient of \(X^2\) gives
\[
 \lambda_1\eta=\nu-1.
\]
If \(\eta=0\), this gives \(\nu=1\), again impossible. If \(\eta\neq0\), then
the constant term gives
\[
 \lambda_1\eta=\nu.
\]
Together with \(\lambda_1\eta=\nu-1\), this is impossible. Therefore there is
no factor of the form \(X+Y+\eta\).

Therefore \(\mathcal C_{\lambda_1,\lambda_2}\) is absolutely irreducible and by Lemma~\ref{lem:irreducible-impossible} \(f_{\lambda_1,\lambda_2}\) is not a permutation polynomial of $\mathbb{F}_{p^2}$.
\end{proof}

In view of Proposition~\ref{prop:lambda2-zero-rdeg3}, from now on we may assume
\[
 \lambda_1\lambda_2\neq0.
\]
Then
\[
 A=-\lambda_1^p\lambda_2\neq0,
 \qquad
 B=\lambda_1\neq0,
\]
and hence the curve \(\mathcal C_{\lambda_1,\lambda_2}\) has
degree \(4\).

\subsection{Factorization of the quartic}

The polynomial $H_{\lambda_1,\lambda_2}(X,Y)$ is symmetric under the involution $(X,Y)\mapsto(Y,X)$ and has degree at most $2$ in each of the variables $X$ and $Y$. If it is reducible, then its components are arranged in one of the following ways: four lines, two conics fixed by the involution, or two conics exchanged by the involution. We analyse these cases by comparing the coefficients in \eqref{eq:ABCDEF}.

\begin{lemma}\label{lem:four-lines}
If
\(\rdeg G_{\lambda_1,\lambda_2}=3\) then $H_{\lambda_1,\lambda_2}$ cannot split into four linear factors.
\end{lemma}

\begin{proof}
Since \(H\) has degree at most \(2\) in each variable, and because of the symmetry \((X,Y)\mapsto(Y,X)\), if $H_{\lambda_1,\lambda_2}$ splits into four linear factors, then 
\[
 H_{\lambda_1,\lambda_2}=A(X+u)(X+v)(Y+u)(Y+v)
\]
for some $u,v$. Put $s_1=u+v$ and $s_2=uv$. Expanding gives
\[
 B=As_1,
 \quad
 C=As_2,
 \quad
 D=As_1^2,
 \quad
 E=As_1s_2,
 \quad
 F=As_2^2.
\]
Thus
\[
 AD=B^2,
 \qquad
 AE=BC,
 \qquad
 AF=C^2.
\]
For the coefficients in \eqref{eq:coefficients-direct}, $AE=BC$ reads
\[
 -\lambda_1^{2p}\lambda_2=\lambda_1(1-\lambda_2^{p+1}),
\]
which is exactly \eqref{eq:E1}. Similarly, the identity \(AD=B^2\) gives
\[ (-\lambda_1^p\lambda_2)(\lambda_1^{p+1}-\lambda_2^{p+1}+1)
 =\lambda_1^2,
\]
that is
\[
\lambda_1^{2p+1}\lambda_2 -\lambda_1^p\lambda_2^{p+2}+\lambda_1^p\lambda_2+\lambda_1^2=0,
\]
which is precisely \eqref{eq:E2}. Hence $N$ and $D$ have a common quadratic factor, and the reduced degree is
$1$, a contradiction.
\end{proof}

\begin{lemma}\label{lem:fixed-conics}
Assume that \(\rdeg G_{\lambda_1,\lambda_2}=3\). Then
\(H_{\lambda_1,\lambda_2}\) cannot factor as a product of two conics both
fixed by the involution \((X,Y)\mapsto(Y,X)\).
\end{lemma}

\begin{proof}
Assume that \(H_{\lambda_1,\lambda_2}\) factors as a product of two conics
which are both fixed by the involution \((X,Y)\mapsto(Y,X)\). Because of the
shape of \(H_{\lambda_1,\lambda_2}\), such a factorization has the form
\[
H_{\lambda_1,\lambda_2}
=
\kappa
\bigl(XY+u(X+Y)+v\bigr)
\bigl(XY+r(X+Y)+s\bigr),
\]
with
\[
 \kappa,u,r,v,s\in\overline{\mathbb{F}}_p,
 \qquad
 \kappa\neq0.
\]
Expanding and comparing coefficients we obtain
\[
 A=\kappa,
\]
\[
 B=\kappa(u+r),
 \qquad
 C=\kappa ur,
\]
\[
 D=\kappa(v+s+2ur),
\]
\[
 E=\kappa(us+rv),
 \qquad
 F=\kappa vs.
\]
Therefore, the polynomial identity
\[
\begin{aligned}
&(u+r)(v+s)(us+rv)
-ur(v+s)^2
-(us+rv)^2        \\
&\qquad
-vs\bigl((u+r)^2-4ur\bigr)=0,
\end{aligned}
\]
yields
\[
 \frac BA
 \left(\frac DA-2\frac CA\right)
 \frac EA
 -
 \frac CA
 \left(\frac DA-2\frac CA\right)^2
 -
 \left(\frac EA\right)^2
 -
 \frac FA
 \left(
 \left(\frac BA\right)^2-4\frac CA
 \right)
 =0.
\]
Multiplying by \(A^3\), this gives
\begin{align}\label{eq:fixed-conics-relation}
0={}&
4ACF-AE^2-B^2F-2BCE+BDE \notag\\
&\qquad
-4C^3+4C^2D-CD^2.
\end{align}

A direct substitution of the explicit coefficients
\[
 A=-\lambda_1^p\lambda_2,
 \qquad
 B=\lambda_1,
 \qquad
 C=1-\lambda_2^{p+1},
\]
\[
 D=\lambda_1^{p+1}-\lambda_2^{p+1}+1,
 \qquad
 E=\lambda_1^p,
 \qquad
 F=-\lambda_1\lambda_2^p
\]
in \eqref{eq:fixed-conics-relation} now gives
\[
\begin{aligned}
&
0=4ACF-AE^2-B^2F-2BCE+BDE
-4C^3+4C^2D-CD^2        \\
&=
-\operatorname{Res}_T(N,D),
\end{aligned}
\]
see \eqref{eq:resultant-ND}.

Therefore
\(
 \operatorname{Res}_T(N,D)=0
\) holds, which is a contradiction with
\(
 \rdeg G_{\lambda_1,\lambda_2}=3
\)
by Proposition \ref{prop:reduced-degree}.
\end{proof}

It remains to consider the case in which the two conics are exchanged by the involution. This
case is the one that produces the third family in the Main Theorem.

\begin{proposition}\label{prop:switched-conics}
Assume that $\rdeg G_{\lambda_1,\lambda_2}=3$, $\lambda_1\lambda_2\neq0$, and let $N=\lambda_1^{p+1}$. Then $H_{\lambda_1,\lambda_2}$ factorizes into two conics exchanged by $(X,Y)\mapsto(Y,X)$, if and only if
there exists $c\in\Fp^{*}$ such that
\[
 \lambda_2=c\lambda_1^3 \qquad \text{ and }\qquad 3cN+1=0.
\]
\end{proposition}
\begin{proof}
Write the factorization as
\begin{equation}\label{eq:switched-factorization}
H_{\lambda_1,\lambda_2}
=
\kappa
(XY+\alpha X+\delta Y+\gamma)
(XY+\delta X+\alpha Y+\gamma).
\end{equation}
Expanding the right-hand side gives
\[
\begin{aligned}
H_{\lambda_1,\lambda_2}
=
\kappa\bigl[
&X^2Y^2
+(\alpha+\delta)(X^2Y+XY^2)
+\alpha\delta(X^2+Y^2)\\
&+(\alpha^2+\delta^2+2\gamma)XY
+\gamma(\alpha+\delta)(X+Y)
+\gamma^2
\bigr].
\end{aligned}
\]
Comparing coefficients with
\[
H_{\lambda_1,\lambda_2}
=
AX^2Y^2+B(X^2Y+XY^2)+C(X^2+Y^2)+DXY+E(X+Y)+F
\]
we obtain
\[
 A=\kappa,
\]
\[
 B=\kappa(\alpha+\delta),
 \qquad
 C=\kappa\alpha\delta,
\]
\[
 D=\kappa(\alpha^2+\delta^2+2\gamma),
\]
\[
 E=\kappa\gamma(\alpha+\delta),
 \qquad
 F=\kappa\gamma^2.
\]
 Put
\[
 S=\alpha+\delta=\frac BA,
 \qquad
 P=\alpha\delta=\frac CA.
\]
Since \(B\neq0\), we have \(S\neq0\). From
\[
 E=A\gamma S
\]
and \(AS=B\), we get
\[
 \gamma=\frac EB.
\]
Therefore, the equality
\[
 F=A\gamma^2
\]
yields
\begin{equation}\label{eq:coeff1}
 B^2F=AE^2.
\end{equation}

It remains to use the coefficient of \(XY\). Since
\[
 \alpha^2+\delta^2
 =
(\alpha+\delta)^2-2\alpha\delta
 =
S^2-2P,
\]
we have
\[
 D=A(S^2-2P+2\gamma).
\]
Substituting
\[
 S=\frac BA,\qquad P=\frac CA,\qquad \gamma=\frac EB
\]
gives
\[
 D
 =
 A\left(
 \frac{B^2}{A^2}
 -2\frac CA
 +2\frac EB
 \right),
\]
and hence
\begin{equation}\label{eq:coeff2}
 ABD=2A^2E+B^3-2ABC.
\end{equation}

Conversely, assume that \eqref{eq:coeff1} and \eqref{eq:coeff2} hold. Define
\[
 \kappa=A,\qquad
 S=\frac BA,\qquad
 P=\frac CA,\qquad
 \gamma=\frac EB.
\]
Let \(\alpha,\delta\in\overline{\Fp}\) be the two roots of
\[
 Z^2-SZ+P=0.
\]
Then
\[
 \alpha+\delta=S=\frac BA,
 \qquad
 \alpha\delta=P=\frac CA.
\]
With these choices, the coefficients of \(X^2Y^2\), \(X^2Y+XY^2\), and
\(X^2+Y^2\) in \eqref{eq:switched-factorization} are respectively
\[
 A,\qquad B,\qquad C.
\]
Also,
\[
 A\gamma(\alpha+\delta)
 =
 A\frac EB\frac BA
 =
 E,
\]
so the coefficient of \(X+Y\) is \(E\). Moreover, \eqref{eq:coeff1} gives
\[
 A\gamma^2
 =
 A\left(\frac EB\right)^2
 =
 F,
\]
so the constant coefficient is \(F\). Finally,
\[
 A(\alpha^2+\delta^2+2\gamma)
 =
 A(S^2-2P+2\gamma),
\]
and using \eqref{eq:coeff2} this is exactly \(D\). Hence
\eqref{eq:switched-factorization} holds over \(\overline{\Fp}\).

Now, substituting \eqref{eq:coefficients-direct} into \eqref{eq:coeff1} gives
\[
 \lambda_1^3\lambda_2^p=\lambda_1^{3p}\lambda_2.
\]
Since $\lambda_1\lambda_2\neq0$, this is equivalent to
\[
 \left(\frac{\lambda_2}{\lambda_1^3}\right)^p
 =\frac{\lambda_2}{\lambda_1^3}.
\]
Thus
\[
 c:=\frac{\lambda_2}{\lambda_1^3}\in\Fp^{*}.
\]
Now substituting $\lambda_2=c\lambda_1^3$ into
\eqref{eq:coeff2} gives
\[
 \lambda_1^3(3cN+1)(c^2N^3-cN^2-1)=0.
\]
As $\lambda_1\neq0$, one of the two factors vanishes. If
$c^2N^3-cN^2-1=0$, then we are in the reduced
degree-one case, a contradiction. Therefore,
$3cN+1=0$ holds. 

Conversely, assume that
\[
 \lambda_2=c\lambda_1^3,\qquad c\in\Fp^{*},\qquad 3cN+1=0,
\]
where \(N=\lambda_1^{p+1}\). We show that the two coefficient relations
\eqref{eq:coeff1} and \eqref{eq:coeff2} hold.

Since \(c\in\Fp\), we have
\[
 \lambda_2^p=c\lambda_1^{3p}.
\]
Therefore
\[
 B^2F
 =
 \lambda_1^2(-\lambda_1\lambda_2^p)
 =
 -\lambda_1^3\lambda_2^p
 =
 -c\lambda_1^{3p+3},
\]
whereas
\[
 AE^2
 =
 (-\lambda_1^p\lambda_2)\lambda_1^{2p}
 =
 -\lambda_1^{3p}\lambda_2
 =
 -c\lambda_1^{3p+3}.
\]
Thus
\[
 B^2F=AE^2,
\]
so \eqref{eq:coeff1} holds.

For the second relation, using
\[
 A=-\lambda_1^p\lambda_2=-c\lambda_1^{p+3},
 \qquad
 B=\lambda_1,
\]
\[
 C=1-\lambda_2^{p+1}=1-c^2N^3,
 \qquad
 D=\lambda_1^{p+1}-\lambda_2^{p+1}+1=N-c^2N^3+1,
\]
and \(E=\lambda_1^p\), we obtain
\[
\begin{aligned}
&ABD-\bigl(2A^2E+B^3-2ABC\bigr)\\
&\qquad =
\lambda_1^3(3cN+1)(c^2N^3-cN^2-1).
\end{aligned}
\]
Since \(3cN+1=0\), this gives
\[
 ABD=2A^2E+B^3-2ABC,
\]
that is, \eqref{eq:coeff2} holds.

So, \(H_{\lambda_1,\lambda_2}\) admits a
factorization of the form \eqref{eq:switched-factorization} over
\(\overline{\Fp}\).
\end{proof}

\begin{proposition}\label{prop:line-square-condition}
Assume $\rdeg G_{\lambda_1,\lambda_2}=3$ and
\[
 \lambda_2=c\lambda_1^3,
 \qquad
 c\in\Fp^{*},
 \qquad
 N=\lambda_1^{p+1},
 \qquad
 3cN+1=0.
\]
Then $f_{\lambda_1,\lambda_2}$ is a permutation polynomial of $\Fpp$ if and only if \[
 -3(1-4c^2N^3)\in\squarep.
\]
\end{proposition}

\begin{proof}
By Proposition \ref{prop:switched-conics}, $H_{\lambda_1,\lambda_2}$ factors as in \eqref{eq:switched-factorization}.
After substituting $\lambda_2=c\lambda_1^3$ and $3cN+1=0$, the two conics can be
written, up to nonzero scalar factors, in the form
\[
 XY+\alpha X+\delta Y+\gamma=0,
 \qquad
 XY+\delta X+\alpha Y+\gamma=0,
\]
\[
 \alpha+\delta=\frac BA,\qquad
 \alpha\delta=\frac CA,\qquad
 \gamma=\frac EB,
\]
where
\[
 A=-\lambda_1^p\lambda_2=\frac{\lambda_1^2}{3},
 \qquad
 B=\lambda_1,
\]
\[
 C=1-\lambda_2^{p+1}=1-c^2N^3=1-\frac N9=\frac{9-N}{9},
\]
and
\[
 \gamma=\frac EB=\frac{\lambda_1^p}{\lambda_1}
 =\lambda_1^{p-1}=\frac N{\lambda_1^2}.
\]
Thus
\[
 \alpha+\delta=\frac3{\lambda_1},
 \qquad
 \alpha\delta=\frac{9-N}{3\lambda_1^2}.
\]
Equivalently, the two numbers
\[
 \lambda_1\alpha,\qquad \lambda_1\delta
\]
are the roots of
\[
 Z^2-3Z+\frac{9-N}{3}=0.
\]
Its discriminant is
\[
 \Theta:=9-\frac{4(9-N)}3=\frac{4N-9}{3}.
\]
Since \(c^2N^3=N/9\), we also have
\[
 \Theta=-3(1-4c^2N^3).
\]
Let \(\rho\) be one root of
\[
 Z^2-3Z+\frac{9-N}{3}=0.
\]
Then the other root is \(3-\rho\).

We now study the affine points $(X,Y)$ of the two conics with coordinates in
\(\mu_{p+1}\times\mu_{p+1}\). It is useful to rescale the variables in order to translate the condition
\(X,Y\in\mu_{p+1}\) into a norm condition. Put
\[
 x=\lambda_1X,\qquad y=\lambda_1Y.
\]
Then \(X,Y\in\mu_{p+1}\) if and only if
\[
 x^{p+1}=y^{p+1}=N.
\]
In these variables, one of the two conic components becomes
\begin{equation}\label{eq:scaled-line-conic}
 xy+\rho x+(3-\rho)y+N=0.
\end{equation}
The other component is obtained by replacing \(\rho\) with \(3-\rho\). Moreover,
\[
 X,Y\in\mu_{p+1}
 \quad\Longleftrightarrow\quad
 x^{p+1}=y^{p+1}=N.
\]

Assume first that \(\Theta\in\squarep\). Then \(\Theta\neq0\), so
\(\rho,3-\rho\in\Fp\) are distinct. We show that in this case
\eqref{eq:scaled-line-conic} has no solution with
\[
 x^{p+1}=y^{p+1}=N,\qquad x\neq y.
\]
From \eqref{eq:scaled-line-conic},
\[
 y=-\frac{N+\rho x}{x+3-\rho}.
\]
The denominator cannot vanish. Indeed, if \(x=\rho-3\), then the numerator would
also have to vanish, so
\[
 N+\rho(\rho-3)=0.
\]
Since \(\rho(3-\rho)=(9-N)/3\), this gives
\[
 N=\frac{9-N}{3},
\]
hence \(N=9/4\), which is equivalent to \(\Theta=0\), a contradiction.

Now impose \(y^{p+1}=N\). Since \(\rho\in\Fp\) and \(x^{p+1}=N\), we have
\(x^p=N/x\). Therefore
\[
 y^p
 =
-\frac{N+\rho N/x}{N/x+3-\rho}
 =
-\frac{N(x+\rho)}{N+(3-\rho)x}.
\]
Hence \(y^{p+1}=N\) is equivalent to
\[
 (x+\rho)(N+\rho x)=(N+(3-\rho)x)(x+3-\rho).
\]
After expansion, using \(\rho+(3-\rho)=3\), this becomes
\[
 (2\rho-3)(x^2+3x+N)=0.
\]
Since \(\Theta\neq0\), we have \(2\rho-3\neq0\), and therefore
\[
 x^2+3x+N=0.
\]
However, substituting \(y=x\) into \eqref{eq:scaled-line-conic} gives precisely
\[
 x^2+3x+N=0.
\]
Since the equation determines \(y\) uniquely, we get \(y=x\). Thus the conic has
no point with \(x^{p+1}=y^{p+1}=N\) off the diagonal. The same argument applies
to the other component.

We also need to check that the denominator of \(G_{\lambda_1,\lambda_2}\) does
not vanish on \(\mu_{p+1}\). Let \(T\in\mu_{p+1}\) and put
\[
 w=\lambda_1T.
\]
Then \(w^{p+1}=N\), and
\[
1+\lambda_1T+\lambda_2T^3=1+w+cw^3.
\]
Since \(c=-1/(3N)\), a zero of the denominator would satisfy
\[
 w^3-3Nw-3N=0.
\]
Raising to the \(p\)-th power and using \(w^p=N/w\), we get
\[
 \frac{N^3}{w^3}-3N\frac Nw-3N=0,
\]
and hence
\[
 N^2-3Nw^2-3w^3=0.
\]
Using again \(w^3=3Nw+3N\), we obtain
\[
 w^2+3w+3-\frac N3=0.
\]
The discriminant of this quadratic is
\[
 9-4\left(3-\frac N3\right)=\frac{4N-9}{3}=\Theta.
\]
Since \(\Theta\in\squarep\), any such \(w\) lies in \(\Fp\). Then
\(w^{p+1}=N\) gives \(w^2=N\). Substituting \(w^2=N\) into
\[
 w^3-3Nw-3N=0
\]
gives
\[
 -2Nw-3N=0,
\]
so \(w=-3/2\), and hence \(N=w^2=9/4\). This would imply \(\Theta=0\),
contrary to \(\Theta\in\squarep\). Thus the denominator has no zero on
\(\mu_{p+1}\).

Therefore, if \(\Theta\in\squarep\), the denominator of \(G_{\lambda_1,\lambda_2}\) is nonzero on
\(\mu_{p+1}\), and the curve $\mathcal{C}_{\lambda_1,\lambda_2}$ has no point
\[
 (X,Y)\in\mu_{p+1}^2,\qquad X\neq Y.
\]
By Proposition~\ref{prop:agw-reduction},
\(f_{\lambda_1,\lambda_2}\) is a permutation polynomial of \(\Fpp\).

Conversely, assume first that \(\Theta=0\). Then \(N=9/4\). Taking
\[
 w=-\frac32
\]
we have \(w^{p+1}=w^2=N\), since \(w\in\Fp\). Moreover
\(c=-1/(3N)=-4/27\), and
\[
 1+w+cw^3
 =
 1-\frac32-\frac4{27}\left(-\frac{27}{8}\right)
 =
0.
\]
Thus the denominator \(G_{\lambda_1,\lambda_2}\) has a zero on \(\mu_{p+1}\), and
\(f_{\lambda_1,\lambda_2}\) is not a permutation polynomial by Proposition \ref{prop:agw-reduction}.

Finally, assume that \(\Theta\) is a nonsquare in \(\Fp\). Then
\(\rho\notin\Fp\), and
\[
 \rho^p=3-\rho.
\]
Choose any \(x\in\Fpp\) with \(x^{p+1}=N\), and define
\[
 y=-\frac{N+\rho x}{x+3-\rho}.
\]
First observe that the denominator is nonzero. Indeed, if \(x=\rho-3\), then
\[
 x^{p+1}
 =
(\rho-3)(\rho^p-3)
 =
(\rho-3)(-\rho)
 =
\rho(3-\rho)
 =
\frac{9-N}{3},
\]
which is not equal to \(N\), since otherwise \(N=9/4\) and \(\Theta=0\).

Now compute the norm of \(y\). Since \(\rho^p=3-\rho\) and \(x^p=N/x\),
\[
 y^p
 =
-\frac{N+(3-\rho)N/x}{N/x+\rho}
 =
-\frac{N(x+3-\rho)}{N+\rho x}.
\]
Thus
\[
 y^{p+1}=N.
\]
So the conic sends the set \(\{x\in\Fpp:x^{p+1}=N\}\) into itself. Its fixed
points satisfy \(y=x\), hence
\[
 x^2+3x+N=0,
\]
which has at most two solutions. Since \(p+1>2\), we can choose
\(x^{p+1}=N\) which is not fixed. For this choice we obtain
\[
 x^{p+1}=y^{p+1}=N,
 \qquad
 x\neq y,
\]
and \eqref{eq:scaled-line-conic} holds. Therefore, with
\[
 X=\frac{x}{\lambda_1},\qquad
 Y=\frac{y}{\lambda_1},
\]
we have
\[
 X,Y\in\mu_{p+1},\qquad X\neq Y,\qquad
H_{\lambda_1,\lambda_2}(X,Y)=0,
\]
and hence \(f_{\lambda_1,\lambda_2}\) is not a
permutation polynomial in this case.

Therefore \(f_{\lambda_1,\lambda_2}\) is a permutation polynomial if
and only if
\[
 \Theta\in\squarep,
\]
and the claim follows.
\end{proof}

Combining \cref{lem:irreducible-impossible,lem:four-lines,lem:fixed-conics,prop:switched-conics,prop:line-square-condition}
shows that in reduced degree $3$ the only possible permutation polynomials are those in the
third family of the Main Theorem. The binomial subcase gives the first family.

\begin{proof}[Proof of \cref{thm:main}]

By Proposition~\ref{prop:agw-reduction}, \(f_{\lambda_1,\lambda_2}\) is a
permutation polynomial of \(\Fpp\) if and only if
\[
1+\lambda_1T+\lambda_2T^3
\]
has no zero on \(\mu_{p+1}\), and
\(
G_{\lambda_1,\lambda_2}(T)
\)
permutes \(\mu_{p+1}\).

We first prove necessity. Assume that \(f_{\lambda_1,\lambda_2}\) is a permutation
polynomial. If \(\lambda_1=0\), Proposition~\ref{prop:binomial-family} gives
case \textup{(i)}. Hence assume \(\lambda_1\neq0\).

If \(\rdeg G_{\lambda_1,\lambda_2}=1\), then
Proposition~\ref{prop:degree-one-family} gives case \textup{(ii)}. The cases
\(\rdeg G_{\lambda_1,\lambda_2}=0\) and
\(\rdeg G_{\lambda_1,\lambda_2}=2\) do not give permutation polynomials by
Proposition~\ref{prop:degree-zero-two}. Thus it remains to consider
\[
 \rdeg G_{\lambda_1,\lambda_2}=3.
\]
By Proposition~\ref{prop:lambda2-zero-rdeg3}, we may assume
\(\lambda_2\neq0\). Hence
\[
 A=-\lambda_1^p\lambda_2\neq0,
 \qquad
 B=\lambda_1\neq0.
\]
By Lemma~\ref{lem:irreducible-impossible}, the curve $\mathcal{C}_{\lambda_1,\lambda_2}$ is not
absolutely irreducible. The possible reducible cases are then restricted by the
factorization lemmas: the curve cannot split into four lines by
Lemma~\ref{lem:four-lines}, and it cannot split into two conics fixed by the
involution by Lemma~\ref{lem:fixed-conics}. Therefore the only admissible
factorization is the one into two conics exchanged by
\((X,Y)\mapsto(Y,X)\). Proposition~\ref{prop:switched-conics} gives
\[
 \lambda_2=c\lambda_1^3,\qquad c\in\Fp^{*},\qquad 3cN+1=0,
 \qquad N=\lambda_1^{p+1}.
\]
Finally, Proposition~\ref{prop:line-square-condition} gives
\[
 -3(1-4c^2N^3)\in\squarep.
\]
Thus case \textup{(iii)} holds.

Conversely, assume that one of the three conditions in the statement holds.
If \textup{(i)} holds, then the conclusion follows from
Proposition~\ref{prop:binomial-family}. If \textup{(ii)} holds, then
Proposition~\ref{prop:degree-one-family} shows that
\(f_{\lambda_1,\lambda_2}\) is a permutation polynomial of \(\Fpp\). If
\textup{(iii)} holds, then Proposition~\ref{prop:line-square-condition}
shows that \(f_{\lambda_1,\lambda_2}\) is a permutation polynomial of
\(\Fpp\).

Thus each of the three conditions is sufficient. This completes the proof.
\end{proof}
If we restrict to the prime field case, our classification has the following by-product.

\begin{corollary}\label{cor:subfield}
Assume that $\lambda_1,\lambda_2\in\Fp$. Then $f_{\lambda_1,\lambda_2}$ is a permutation
polynomial of $\Fpp$ if and only if one of the following holds:
\begin{enumerate}[label=\rm(\roman*),leftmargin=1.5em]
\item
\[
 \lambda_1=0,
 \qquad
 p\equiv1\pmod3,
 \qquad
 \lambda_2^2\neq1;
\]
\item $\lambda_1\neq0$,
\[
 \lambda_2^2-\lambda_1\lambda_2-1=0,
 \qquad
 1-4\lambda_2^2\in\squarep;
\]
\item $\lambda_1\neq0$,
\[
 \lambda_1+3\lambda_2=0,
 \qquad
 -3(1-4\lambda_2^2)\in\squarep.
\]
\end{enumerate}
\end{corollary}

\begin{proof}
If $\lambda_1,\lambda_2\in\Fp$, then $\lambda_1^{p+1}=\lambda_1^2$ and
$\lambda_2^{p+1}=\lambda_2^2$. The binomial case becomes (i).
In the second family of \cref{thm:main}, we have
$c=\lambda_2/\lambda_1^3$ and $N=\lambda_1^2$. The equation
$c^2N^3-cN^2-1=0$ becomes
\[
 \lambda_2^2-\lambda_1\lambda_2-1=0,
\]
and the square condition becomes $1-4\lambda_2^2\in\squarep$. In the third
family, $3cN+1=0$ becomes
\[
 3\frac{\lambda_2}{\lambda_1^3}\lambda_1^2+1=0,
\]
that is $\lambda_1+3\lambda_2=0$. The square condition becomes
$-3(1-4\lambda_2^2)\in\squarep$.
\end{proof}

\begin{remark}
Although permutation rational functions of degree $3$ are classified up to
equivalence, the alternative proof based on that classification is not really shorter than the geometric argument via algebraic curves. Indeed, in the degree-$3$ case one still has to analyse in detail the possible M\"obius equivalences and impose the additional structural constraints satisfied by our map $G_{\lambda_1,\lambda_2}$. As a consequence, the resulting case-by-case analysis is at least as involved as the argument
based on the auxiliary curve $\mathcal{C}_{\lambda_1,\lambda_2}$.
\end{remark}

\begin{remark}
It is natural to ask whether the permutation polynomials classified in
\cref{thm:main} give rise to different equivalence classes, or whether they all
belong to a single class.

Two maps \(F,G:\Fpp\to\Fpp\) are said to be EA-equivalent if there exist
affine permutations \(L_1,L_2\) of \(\Fpp\), and an affine map \(L_3\), such that
\[
 G=L_1\circ F\circ L_2+L_3.
\]
This is one of the standard equivalence relations used for vectorial functions
over finite fields; see, for instance, \cite{CarletCharpinZinoviev1998,PottZhou}.
A standard invariant under EA-equivalence is the differential spectrum. For a
map \(F:\Fpp\to\Fpp\), this is the multiset of the numbers
\[
 \#\{x\in\Fpp: F(x+a)-F(x)=b\},
 \qquad a\in\Fpp^{*},\ b\in\Fpp.
\]
Thus, if two functions have different differential spectra, then they cannot be
EA-equivalent. The converse is not true in general, so the differential spectrum
does not give a complete classification of EA-classes.

We performed MAGMA computations for small primes using the differential spectrum
as an EA-invariant. In the case of coefficients in \(\Fpp\), these computations
already show that the family in \cref{thm:main} does not form a single
EA-equivalence class. For instance, for \(p=5\), writing
\(\mathbb F_{25}=\Fp(\beta)\) with \(\beta^2=2\), the admissible pairs
\((\lambda_1,\lambda_2)\) split into at least two different differential
spectra. The same happens for \(p=7\), writing
\(\mathbb F_{49}=\Fp(\beta)\) with \(\beta^2=3\). Hence the permutation
polynomials classified in \cref{thm:main} are not all EA-equivalent.

In the subfield case \(\lambda_1,\lambda_2\in\Fp^*\), the computations suggest a
more rigid behaviour. For \(p=11\), the admissible pairs split into two
EA-classes:
\[
 \{(3,10),(8,1)\}
\]
and
\[
 \{(1,4),(1,8),(5,2),(6,9),(10,3),(10,7)\}.
\]
For \(p=13\), they again split into two EA-classes:
\[
 \{(3,12),(10,1)\}
\]
and
\[
 \{(4,3),(6,10),(7,3),(9,10)\}.
\]
The same pattern occurs in the computations for \(p=17\) and \(p=19\): there is
one class containing the two special pairs
\[
 (-3,1)\qquad\text{and}\qquad (3,-1),
\]
and another class containing all the remaining admissible pairs.
\end{remark}

\begin{conjecture}
Assume that \(\lambda_1,\lambda_2\in\Fp^*\) . For every odd prime \(p\ge 11\), the
permutation polynomials classified in \cref{thm:main} fall into exactly two
EA-equivalence classes. One class consists of the two special pairs
\[
 (-3,1)\qquad\text{and}\qquad (3,-1),
\]
while the other contains all the remaining admissible pairs.
\end{conjecture}

\section*{Acknowledgments}
The authors thank the Italian National Group for Algebraic and Geometric Structures and their Applications (GNSAGA—INdAM)
which supported the research.

\end{document}